\documentclass[12pt]{article}

\usepackage{amsmath}
\usepackage{amsfonts}
\usepackage{amssymb}
\makeatletter

\newcommand{\Rmnum}[1]{\expandafter\@slowromancap\romannumeral #1@}
\makeatother

\begin{document}

\newtheorem{theorem}{Theorem}
\newtheorem{cor}[theorem]{Corollary}
\newtheorem{guess}{Conjecture}
\newtheorem{claim}[theorem]{Claim}
\newtheorem{lemma}[theorem]{Lemma}

\makeatletter
  \newcommand\figcaption{\def\@captype{figure}\caption}
  \newcommand\tabcaption{\def\@captype{table}\caption}
\makeatother

\newtheorem{acknowledgement}[theorem]{Acknowledgement}

\newtheorem{question}{Question}

\newenvironment{proof}{\noindent {\bf
Proof.}}{\hfill$\blacksquare$\medskip}
\newcommand{\remark}{\medskip\par\noindent {\bf Remark.~~}}
\newcommand{\pp}{{\it p.}}
\newcommand{\de}{\em}
\newcommand{\mad}{\rm mad}

\newcommand{\qed}{\hfill$\blacksquare$\medskip}

\title{The Strongly Antimagic labelings of Double Spiders}

\author{
Fei-Huang Chang\thanks{Grant number: MOST 104-2115-M-003-008-MY2}\\
   Division of Preparatory Programs
   for  Overseas Chinese Students\\
   National Taiwan Normal University\\
	 New Taipei City, Taiwan\\
    {\tt cfh@ntnu.edu.tw}
    \and
    Pinhui Chin  \\
    Department of Mathematics\\
    Tamkang University\\
    New Taipei City, Taiwan\\
    {\tt encorex63447@gmail.com}
   \and
Wei-Tian Li\thanks{Grant number: MOST-105-2115-M-005-003-MY2}\\
   Department of Applied Mathematics\\
   National Chung-Hsing University\\
   Taichung City, Taiwan\\
   {\tt weitianli@nchu.edu.tw}
   \and
Zhishi Pan \thanks{Corresponding Author}\\
   Department of Mathematics\\
   Tamkang University\\
   New Taipei City, Taiwan\\
   {\tt zhishi.pan@gmail.com}
  }
\maketitle

\begin{abstract}
A graph $G=(V,E)$ is strongly antimagic, if there is a bijective mapping $f: E \to \{1,2,\ldots,|E|\}$
such that for any two vertices $u\neq v$, not only $\sum_{e \in E(u)}f(e) \ne \sum_{e\in E(v)}f(e)$
 and also $\sum_{e \in E(u)}f(e) < \sum_{e\in E(v)}f(e)$ whenever $\deg(u)< \deg(v) $,
where $E(u)$ is the set of edges incident to $u$.
In this paper, we prove that double spiders, the trees contains exactly two vertices of degree at least 3, are strongly antimagic.

\end{abstract}

%Keywords: Anti-magic, strongly anti-magic, labeling.

%%%%%%%%%%%%%%%%%%%%%%%%%%%%%%%%%%%%%%%%%%%%%%%%%%%%%%%%%%%%%%%%%%%%
\section{Introduction}

 Suppose  $G=(V,E)$ is a connected, finite, simple graph and   $f: E \to \{1,2,\ldots,
 |E|\}:=[|E|]$ is a bijection. For each vertex $u$ of $V$, let $E(u)$ be the set of edges incident to $u$, and the {\em vertex-sum} $\varphi_f$ at $u$ is defined as  $\varphi_f(u)=\sum_{e \in E(u)}f(e)$. The degree of $u$, denoted by $\deg(u)$, is the capacity of $E(u)$, i.e. $\deg(u)=|E(u)|$ and the leaf set is defined by $V_1=\{u|\deg(u)=1, u\in V\}$.
If $\varphi_f(u) \ne \varphi_f(v)$ for any two distinct vertices $u$ and $v$ of $V$, then $f$ is called an {\em antimagic labeling} of $G$.

The problem of finding antimagic labelings of graphs was introduced by Hartsfield and Ringel \cite{HR1990} in $1990$.
They proved that some special families of graphs, such as {\em paths, cycles, complete graphs}, are antimagic and put two conjectures. The conjectures have received much attention, but both conjectures
remain open.

\begin{guess} \cite{HR1990}\label{g1}
Every connected graph with order at least $3$ is antimagic.
\end{guess}

The most significant progress of Conjecture \ref{g1} is a result of Alon, Kaplan, Lev, Roditty, and
Yuster \cite{AKLRY2004}. They proved that a graph $G$ with minimum degree $\delta(G)\geq c\log |V|$ for a constant $c$ or with maximum degree $\Delta(G)\geq |V|-2$ is antimagic. They also proved that complete partite graph other than $K_2$ is antimagic.

Cranston \cite{Cra2009} proved that for $k \ge 2$, every $k$-regular bipartite graph
is antimagic. For non-bipartite regular graphs, B\'{e}rczi, Bern\'{a}th, Vizer \cite{BBV2015} and
Chang, Liang, Pan, Zhu \cite{CLPZ2016} proved independently that every regular graph is antimagic.

\begin{guess} \cite{HR1990}\label{g2}
Every  tree other than $K_2$ is antimagic.
\end{guess}

For Conjecture \ref{g2}, Kaplan, Lev, and Roditty \cite{KLR2009}
and Liang, Wong, and Zhu \cite{LWZ2012} showed that a tree with at most one vertex of degree $2$ is antimagic.
Recently, Shang \cite{S2015} proved that a special family of trees, spiders, is antimagic.
A {\em spider} is a tree formed from taking a set of disjoint paths and identifying one endpoint of each path together.
Huang, in his thesis~\cite{H2015}, also proved that spiders are antimagic.
Moreover, the antimagic labellings $f$ given in~\cite{H2015} have the property:  $\deg(u)<\deg(v)$ implies $\varphi_f(u)<\varphi_f(v)$.
Given a graph $G$, if there exists an  antimagic labeling $f$ satisfying the above property,
the $f$ is called a {\em strongly antimagic labeling} of $G$.
A graph $G$ is called {\em strongly antimagic}  if it has a strongly antimagic labeling.

Finding a strongly antimagic labeling on a graph $G$ enables us to find an antimagic labeling
of the supergraph of $G$.
Let us describe such inductive method in Lemma \ref{lm1}
which is extracted from the ideas in~\cite{H2015}.
For a graph $G$, let $V_k$ be the set of vertices of degree $k$ in $V(G)$.
Assume that $V_1=\{v_1,v_2,...,v_i\}$, then we define $G\bigoplus V_1'=(V(G)\cup V_1',E(G)\cup E')$,
where $V_1'=\{v_1',v_2',...,v_i'\}$ and $E'=\{v_1v_1',v_2v_2',...,v_iv_i'\}$.

\begin{lemma} \label{lm1}
  For any connected graph $G$ with $V_1\neq\varnothing$, if $G$ is strongly antimagic, then $G\bigoplus V_1'$ is strongly antimagic.
\end{lemma}
\begin{proof}
The proof of this lemma is similar to the proof of a corollary in \cite{H2015}. Let $f$ be a strongly antimagic labeling of $G$ and $V_1=\{v_1,v_2,...,v_i\}$ with $\varphi_f(v_1)<\varphi_f(v_2)<...<\varphi_f(v_i)$.
We construct a bijective mapping $f':E(G\bigoplus V_1')\rightarrow [|E(G)|+i]$ as following.

\[
f'(e)=\left\{
          \begin{array}{ll}
            j, & \hbox{if $e=v_jv_j'\in E',1\leq j\leq i$;} \\
            f(e)+i, & \hbox{if $e\in E$.}
          \end{array}
        \right.
\]

For any vertices $u\in V(G)-V_1 $, $v_j\in V_1$, and $v_{j'}\in V_1'$, the vertex sums
under $f'$ are $\varphi_{f'}(u)=\varphi_f(u)+i\deg(u)$, $\varphi_{f'}(v_j)=\varphi_f(v_j)+j$,
and $\varphi_{f'}(v_j')=j$.
By some calculations and comparisons,
it is clear that $f'$ is a strongly antimagic labeling of $G\bigoplus V_1'$.
\end{proof}

A {\em double spider} is a tree which contains exactly two vertices of degree greater than 2.
It also can formed by first taking two sets of disjoint paths and one extra path,
and then identifying an endpoint of each path in the two sets to the two endpoints of the extra path, respectively. 
In this paper, we manage to solve Conjecture~\ref{g2} for double spiders.
We have a stronger result:

\begin{theorem} \label{main thm}
  Double spiders are strongly antimagic.
\end{theorem}

The rest of the paper is organized as follows.
In Section 2, we give some reduction methods 
and classify the double spiders into four types.
For the four types of the double spiders,
we will prove that they are strongly antimagic by giving the labeling rules
in four different lemmas.
Hence, to prove our main theorem, it suffices to prove the lemmas.
The proofs of the lemmas are presented in Section 3.
However, we will only give the labeling rules and show the 
strongly antimagic properties for degree one and degree two vertices.
For the comparisons between other vertices, we put all the details in Appendix.
Some concluding remarks and problems will be proposed in Section 4.

\section{Main Results}

Given a double spider, we decompose its edge set into three subsets:
The core path $P^{core}$, $L$, and $R$, where
$P^{core}$ is the unique path connecting the two vertices of degree at least three,
$L$ consists of paths with one endpoint of each path identified to an endpoint of $P^{core}$,
and $R$ consists of paths with one endpoint of each path identified to the other endpoint of $P^{core}$.
We denote the endpoints of $P^{core}$ by $v_l$ and $v_r$, respectively.
Conventionally, we assume $L$ contains at least as many paths as $R$, hence $\deg(v_l)\ge \deg(v_r)$.
See Figure $1$ as an illustration.
Note that two double spiders are isomorphic if their $L$ sets, $R$ sets, and the core paths are isomorphic.
From now on, we denote a double spider by $DS(L,P^{core},R)$.
The complexity of finding an antimagic labeling of a double spider depends on the number of the paths and their lengths composing the double spider.
So let us begin with reducing the number of paths of length one in $L$ and $R$.

\begin{figure}[h]
\begin{center}
\begin{picture}(220,90)
\put(50,70){\circle*{4}}
\put(30,70){\circle*{4}}
\put(10,70){\circle*{4}}
\put(-10,70){\circle*{4}}
\put(50,50){\circle*{4}}
\put(30,50){\circle*{4}}
\put(10,50){\circle*{4}}
\put(-10,50){\circle*{4}}
\put(-30,50){\circle*{4}}
\put(50,10){\circle*{4}}
\put(50,-10){\circle*{4}}
\put(70,30){\circle*{4}}
\put(68,35){$v_l$}
\put(120,35){$v_r$}
\put(50,70){\line(-1,0){60}}
\put(50,50){\line(-1,0){80}}
\put(70,30){\line(-1,1){20}}
\put(70,30){\line(-1,2){20}}
\put(70,30){\line(-1,-1){20}}
\put(70,30){\line(-1,-2){20}}
\put(100,35){\oval(90,20)}
\put(100,10){$P^{core}$}
\put(90,30){\circle*{4}}
\put(110,30){\circle*{4}}
\put(130,30){\circle*{4}}
\put(70,30){\line(1,0){60}}
\put(150,70){\circle*{4}}
\put(150,70){\line(1,0){40}}
\put(130,30){\line(1,1){20}}
\put(130,30){\line(1,2){20}}
\put(130,30){\line(1,-1){20}}
\put(170,70){\circle*{4}}
\put(190,70){\circle*{4}}
\put(150,50){\circle*{4}}
\put(150,10){\circle*{4}}
\put(170,10){\circle*{4}}
\put(150,10){\line(1,0){20}}
\end{picture}
\begin{picture}(120,80)
\put(0,70){$R:$}
\put(110,70){\circle*{4}}
\put(50,70){\circle*{4}}
\put(70,70){\circle*{4}}
\put(90,70){\circle*{4}}
\put(70,60){\circle*{4}}
\put(90,60){\circle*{4}}
\put(110,60){\circle*{4}}
\put(110,50){\circle*{4}}
\put(90,50){\circle*{4}}
\put(110,70){\line(-1,0){60}}
\put(110,60){\line(-1,0){40}}
\put(110,50){\line(-1,0){20}}

\put(0,30){$L:$}
\put(110,30){\circle*{4}}
\put(30,30){\circle*{4}}
\put(50,30){\circle*{4}}
\put(70,30){\circle*{4}}
\put(90,30){\circle*{4}}
\put(10,20){\circle*{4}}
\put(30,20){\circle*{4}}
\put(50,20){\circle*{4}}
\put(70,20){\circle*{4}}
\put(90,20){\circle*{4}}
\put(110,20){\circle*{4}}
\put(110,10){\circle*{4}}
\put(90,10){\circle*{4}}
\put(110,0){\circle*{4}}
\put(90,0){\circle*{4}}
\put(110,30){\line(-1,0){80}}
\put(110,20){\line(-1,0){100}}
\put(110,10){\line(-1,0){20}}
\put(110,00){\line(-1,0){20}}
\end{picture}
 \end{center}
\caption{A double spider $DS(L,P^{core},R)$.}\label{fig1}
\end{figure}
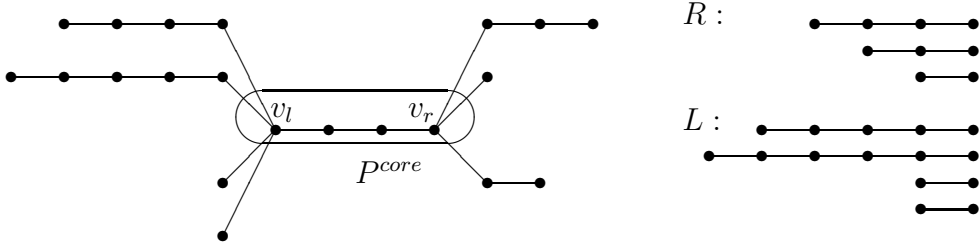

\begin{lemma} \label{Right minus 1}\label{casebasic}

Suppose $G=DS(L,P^{core},R)$ contains some paht $P$ of length one  in $L\cup R$. Assume at least one of the following conditions holds.
\begin{description}
\item{(1)} $P\in R$, $\deg_G(v_l)\geq \deg_G(v_r)>3$, and $DS(L,P^{core},R\setminus \{P\})$ is strongly antimagic,
\item{(2)} $P\in L$, $\deg_G(v_l)>\deg_G(v_r)\geq 3$, and $DS(L\setminus \{P\},P^{core},R)$ has a strongly antimagic labeling $f$ with $\varphi_f(v_l)>\varphi_f(v_r)$.
\end{description}
Then $G$ is strongly antimagic.

 \end{lemma}

\begin{proof} We only prove (1), since the proof of (2) is analogous.
  Let $f$ be a strongly antimagic labeling on $G'=DS(L,P^{core},R\setminus \{P\})$, and $P=v_rv$. We create a bijective mapping $f^*$ from $E(G)$ to $[|E(G)|]$ on $G$ by
\[f^*(e)=\left\{
 \begin{array}{ll}
 1, & \hbox{if $e=v_rv$;} \\
 f(e)+1, & \hbox{if $e\in E(G')$}.
 \end{array}
\right.\]
Since $\deg_{G'}(v_r)<\deg_{G'}(v_\ell)$, we have $\varphi_f(v_r)<\varphi_f(v_l)$, and \begin{eqnarray*}
                                                 \varphi_{f^*}(v_r) &=& \varphi_f(v_r)+(\deg_G(v_r)-1)+f^*(v_rv)\\
                                                  &=& \varphi_f(v_r)+\deg_G(v_r) \\
                                                  &<& \varphi_{f}(v_l)+\deg_G(v_l)=\varphi_{f^*}(v_l).
                                               \end{eqnarray*}
 It is clear that $f^*$ is a strongly antimagic labeling of $DS(L,P^{core},R).$
\end{proof}

An {\em odd path (rest. even path)} is a path of odd (even) length.
Now suppose $R$ consists of $a$ odd paths and $b$ even paths,
and $L$ consists of $c$ odd paths with length greater than one, $d$ even paths,
and $t$ odd paths of length one.
By means of the following lemmas, Theorem~\ref{main thm} will be proved.

\begin{lemma} \label{deqL equals 3}
If $\deg(v_l)=\deg(v_r)=3$ then $DS(L,P^{core},R)$ is strongly antimagic.
 \end{lemma}

\begin{lemma} \label{R has two P1}
If $\deg(v_l)>\deg(v_r)\geq 3$, $b=0$, and $R$ has no odd path of length at least $3$, then $DS(L,P^{core},R)$ is strongly antimagic.
 \end{lemma}

\begin{lemma} \label{R has odd path}
If $\deg(v_l)>\deg(v_r)\geq 3$, $b=0$, and $R$ has at least one odd path of length at least $3$, then $DS(L,P^{core},R)$ is strongly antimagic.
 \end{lemma}

\begin{lemma} \label{R has even path}
If $\deg(v_l)>\deg(v_r)\geq 3$ and $b\geqslant 1$, then $DS(L,P^{core},R)$ is strongly antimagic.
\end{lemma}

\noindent{\bf Proof of Theorem~\ref{main thm}.}
By Lemma \ref{deqL equals 3}, \ref{R has two P1}, \ref{R has odd path}, and \ref{R has even path},
the remaining case we need to show is $\deg(v_l)=\deg(v_r)\geq 4$.
For such a double spider $DS(L,P^{core},R)$,
let $h$ be the minimum length of a path in $L\cup R$. 
Without loss of generality, we assume there is a $P_h$ in $R$.
Consider the double spider $DS(L',P^{core},R')$ that is obtained by
recursively deleting the leaf sets of $DS(L,P^{core},R)$ and of the resulting graphs $h-1$ times.
According to Lemma \ref{lm1}, we only need to show that $DS(L',P^{core},R')$
is strongly antimagic.
It is clear that $R'$ contains a path $P$ of length one.
By Lemma \ref{Right minus 1}, it is sufficient to show $G^{*}=DS(L',P^{core},R'\setminus \{P\})$ is strongly antimagic. Now $\deg_{G^{*}}(v_l)>\deg_{G^{*}}(v_r)$, by Lemma \ref{R has two P1}, \ref{R has odd path}, and \ref{R has even path}, $G^{*}$ is strongly antimagic.
\qed

\section{Proofs of the Remaining Lemmas}

In this section, we are going to prove the Lemmas in last section.
We will give the rules to label the double spiders in each proof.
However, part of the work of checking the strongly antimagic property is moved to Appendix
because of the tedious and complicated calculations.

To achieve the goal, we need to give all edges and vertices the informative names.
Here we use $P_h$ to denote a path with length $h$, i.e. $P_h=u_0u_1u_2,...,u_h$,
which is not the common way but is helpful for us to simply the notation in our proof.
In addition, for paths of the same length in $L$ (or $R$), we can interchange the labelings on the edges of one paths with those of another. Thus, only the length of a path matters,
and we use the same notation to represent paths of the same length in $L$ or $R$.
Now, let $DS(L,P^{core},R)$ be a double spider with path
parameters $a$, $b$, $c$, $d$, and $t$ defined in Section 2,
and let $s$ be the length of $P^{core}$. Then we name the vertices and edges on the paths as follows:

$P^{core} = v_1v_2\cdots v_{s+1}$ with $v_1=v_\ell,v_{s+1}=v_r$ and $e_i=v_iv_{i+1}$.

$R= \{P_{2x_1+1},P_{2x_2+1},...,P_{2x_a+1},P_{2y_1},P_{2y_2},...,P_{2y_b}\}$ with  $y_1\leq \cdots\leq y_b.$

\begin{itemize}

\item  $P_{2x_i+1}=v_rv^{r, odd}_{i,1}v^{r, odd}_{i,2}\cdots v^{r, odd}_{i,2x_i+1}$
with $e^{r, odd}_{i,j}=v^{r, odd}_{i,j-1}v^{r, odd}_{i,j}$ and $e^{r,odd}_{i,1}=v_rv^{r, odd}_{i,1}$.   
\item  $P_{2y_i} =v_rv^{r,even}_{i,1}v^{r,even}_{i,2}\cdots v^{r,even}_{2y_i}$
with  $e^{r, even}_{i,j}=v^{r,even}_{i,j-1}v^{r,even}_{i,j}$ and $e^{r,even}_{i,1}=v_rv^{r, even}_{i,1}$.  
\end{itemize}

$L = \{P_{2w_1+1},...,P_{2w_c+1},P_{2z_1},...,P_{2z_d},P^1_1,P^2_1,...,P^t_1\}$ with $w_i\geq 1 $.

\begin{itemize}

\item  $P_{2w_i+1}=v_l v^{l, odd}_{i,2w_i+1} v^{l, odd}_{i,2w_i}\cdots v^{l, odd}_{i,1}$
with $e^{l, odd}_{i,j-1}=v^{l, odd}_{i,j} v^{l, odd}_{i,j-1}$ and $e^{l, odd}_{i,2w_i+1}=v_lv^{l,odd}_{i,2w_i+1}$. %$1\leq i\leq c. $
\item  $P_{2z_i} =v_lv^{l,even}_{i,2z_i}v^{l,even}_{i,2z_i-1}\cdots v^{l,even}_{i,1}$
with $e^{l, even}_{i,j-1}=v^{l,even}_{i,j}v^{l,even}_{i,j-1}$ and $e^{l, even}_{i,2z_i}=v_lv^{l,even}_{i,2z_i}$.
\item  $P^i_1 = v_lv^i_1$ with $e^i=v_lv^i_1$ $1\leq i\leq t$.

\end{itemize}

A vertex (resp. edge) denoted as $v_{i,j}^{r,odd}$ (resp. $e_{i,j}^{l,even}$) means that it is the $j$th vertex (resp. edge) of the $i$th odd (resp. even) path in $R$ (resp. $L$).
Observe that the index $j$ of an edge of a path in $R$ is increasing from $v_r$ to the leaf of the path, but the index of that in $L$ is reverse.
An edge of a path is called an {\em odd} (or {\em even}) {\em edge} if the index $j$ of the edge is odd (or even).
Define the following quantities for the total number of odd (even) edges in some odd (even) paths. (The summation is zero if $i=0$. )
\[
\begin{array}{lll}
A^{odd}_i=\sum_{k=1}^{i}(x_k+1),& A^{even}_i=\sum_{k=1}^{i}x_k,& B^{odd}_i=B^{even}_i=\sum_{k=1}^{i}y_k,\\
C^{odd}_i=\sum_{k=1}^{i}(w_k+1),& C^{even}_i=\sum_{k=1}^{i}w_k,& D^{odd}_i=D^{even}_i=\sum_{k=1}^{i}z_k.
\end{array}
\]
Let $A^{all}=A^{odd}_a+A^{even}_a,\ldots,D^{all}=D^{odd}_d+D^{even}_d$.
Then the total number of edges $m=A^{all}+B^{all}+s+C^{all}+D^{all}+t$.
We use $[n]_{odd}$ and $[n_1,n_2]_{odd}$ to denote the set of all odd integers in $[n]$ and the set of all odd integers in $\{n_1,n_1+1,\ldots, n_2\}$, respectively.
The definitions of $[n]_{even}$ and $[n_1,n_2]_{even}$ are similar.

Let us begin with Lemma \ref{R has odd path}, which is the simplest one.

\noindent{\bf Proof of Lemma \ref{R has odd path}.} 
We construct a bijective mapping $f$ by assigning $1,2,\ldots,m$ to the edges accordingly in the following steps. 
Some steps can be skipped if no such edges exist. Without loss of generality, $x_a\ge 1$.

\noindent{\bf Step 1.} Label the odd edges of the odd paths in $R$ by 
\begin{center}
$f(e^{r,odd}_{i,j})=
\left\{
\begin{array}{ll}
A^{odd}_{i-1}+\frac{j+1}{2},&  \hbox{for } i\in [a-1]\hbox{ and }j\in[2x_i+1]_{odd}\\ 
A^{odd}_{a-1}+\frac{j-1}{2},& \hbox{for } i=a\hbox{ and }j\in[2,2x_a+1]_{odd}\end{array}
\right.
$
\end{center}
We label the edge $e^{r,odd}_{a,1}$ later in order to ensure that the vertex sum at $v_r$ later is large enough.
Recall $A_0^{odd}=0$.

\noindent{\bf Step 2.} If $c\ge 1$, for $i\in[c]$ and $j\in[2w_i]_{odd}$, label the odd edges of the odd paths in $L$ by 
\begin{center}
$f(e^{l,odd}_{i,j})=A^{odd}_a-1+C^{odd}_{i-1}-(i-1)+\frac{j+1}{2}$.
\end{center}
We also leave the $c$ edges $e^{l,odd}_{i,2w_i+1}$ for $ 1\leq i\leq c$ to enlarge the vertex sum at $v_l$. 

\noindent{\bf Step 3.} If $s\ge 4$, label the edges of $P^{core}$ by,  
\begin{equation*}
f(e_j)=A^{odd}_a-1+C^{odd}_c-c+\left\{ 
\begin{array}{ll}
\frac{s-j}{2}, \hbox{ for }j\in[2, s-2]_{even}, & \hbox{when $s$}\hbox{ is even}. \\
\frac{j-1}{2}, \hbox{ for }j\in[3, s-2]_{odd}, & \hbox{when $s$}\hbox{ is odd}.
\end{array}
\right.
\end{equation*}
In this step, we have labeled $s_1=\lfloor\frac{|s-2|}{2}\rfloor$ edges on the core path $P_s$.

\noindent{\bf Step 4.} If $d\ge 1$, for $i\in[d]$ and $j\in[2z_i]_{odd}$, 
label the odd edges of the even paths in $L$ by 
\begin{center}
$f(e^{l,even}_{i,j})=A^{odd}_a-1+C^{odd}_c-c+s_1+D^{odd}_{i-1}+\frac{j+1}{2}$.
\end{center}

\noindent{\bf Step 5.} If $t\ge 1$, for $i\in[t]$, label the paths of length one in $L$ by  %($e^i$) 
\begin{center}
$f(e^i)=A^{odd}_a-1+C^{odd}_c-c+s_1+D^{odd}_d+i$.
\end{center}

\noindent{\bf Step 6.} For $i\in[a]$ and $j\in[2x_i]_{even}$, label the even edges of the odd paths in $R$ by 
\begin{center}
$f(e^{r,odd}_{i,j})=A^{odd}_a-1+C^{odd}_c-c+s_1+D^{odd}_d+t+A^{even}_{i-1}+\frac{j}{2}$.
\end{center}

\noindent{\bf Step 7.} If $c\ge 1$, for $i\in[c]$ and $j\in[2w_i]_{even}$, label the even edges of the odd paths in $L$ by
\begin{center}
$f(e^{l,odd}_{i,j})=A^{all}-1+C^{odd}_c-c+s_1+D^{odd}_d+t+C^{even}_{i-1}+\frac{j}{2}$.
\end{center}

\noindent{\bf Step 8.} If $s\ge 2$, label the edges in $P^{core}$ by  
\begin{center}
$f(e_j)=A^{all}-1+C^{all}-c+s_1+D^{odd}_d+t+\left\{
                                                                    \begin{array}{lll}
                                                                      \frac{s+1-j}{2},&\hbox{for }  j\in [s]_{odd}, & \hbox{when } s\hbox{ is even}. \\
                                                                      \frac{j}{2},&\hbox{for }  j\in [s]_{even}, & \hbox{when } s\hbox{ is odd}.
                                                                    \end{array}
\right.
$
\end{center}
We now have $s_2$ unlabeled edges on $P^{core}$, where 
$s_2=1$, if $s=1$ or $s$ is even, otherwise $s_2=2$.

\noindent{\bf Step 9.} If $d\ge 1$, for $i\in[d]$ and $j\in[2z_i]_{even}$, label the even edges of the even paths in $L$ by
\begin{center}
$f(e^{l,even}_{i,j})=A^{all}-1+C^{all}-c+s-s_2+D^{odd}_d+t+D^{even}_{i-1}+\frac{j}{2}$.
\end{center}

\noindent{\bf Step 10.} Label the edge $e^{r,odd}_{a,1}$ by
\begin{center}
$f(e^{r,odd}_{a,1})=A^{all}-1+C^{all}-c+s-s_2+D^{all}+t+1=m-c-s_2.$
\end{center}

\noindent{\bf Step 11.} If $c\ge 1$, for $i\in[c]$, label the edges $e^{l,odd}_{i,2w_i+1}$ by 
\begin{center}
$f(e^{l,odd}_{i,2w_i+1})=m-c-s_2+i$. 
\end{center}

\noindent{\bf Step 12.} Label the remaining edges in $P^{core}$ by the following rules:
If $s=1$ or $s$ is even, then let $f(e_s)=m$; otherwise, let $f(e_1)=m-1$ and $f(e_s)=m$.

We prove that  $f$ is strongly antimagic:

{\bf Claim:} $\varphi_f(u)> \varphi_f(v)$ for any $u\in V_2$ and $v\in V_1$.

Observe that, at Step 5, every pendant edge has been labeled, 
and for a vertex $u$ in $V_2$, there is an unlabeled edge in $E(u)$. 
This guarantees that $\varphi_f(u)>\varphi_f(v)$ for every vertex $v$ in $V_1$.

{\bf Claim:} $\varphi(u)$ are all distinct for $u\in V_2$.

For any two vertices $u'$ and $u''$ in $V_2$, let $E(u')=\{e^1_{u'}, e^2_{u'}\}$ and $E(u'')=\{e^1_{u''}, e^2_{u''}\}$.
Assume $f(e^1_{u'})<  f(e^2_{u'})$ and $f(e^1_{u''})<  f(e^2_{u''})$.
Our labeling rules give that 
if $f(e^1_{u'})\le  f(e^1_{u''})$, then $f(e^2_{u'})\le  f(e^2_{u''})$, and at least 
one of the inequalities is strict. This guarantees that 
$\varphi_f(u)$ are distinct for all $u\in V_2$.  

For  $\varphi_f(v_l)>\varphi_f(v_r)>\varphi_f(u)$ for any $u\in V_2$, see Appendix. 
\qed

The next is the proof of Lemma~\ref{R has even path}.

\noindent{\bf Proof of Lemma~\ref{R has even path}.} 
We use similar concepts of the proof of Lemma \ref{R has odd path} to create a bijection from $E$ to $[m]$.
However, the rules will be a little more complicated than those in Lemma \ref{R has odd path}.
In our basic principles, for each path in $L$ or $R$, 
edges of the same parity as the pendent edge on the same path should be labeled first in general, 
except for some edges incident to $v_l$ or $v_r$; and edges of different parities to the pendent edge on the same path
will always be labeled after all pendent edges have been labeled (so that the vertex sum at a vertex of degree two 
can be greater than the vertex sum of a pendent vertex).   
Thus, when $L$ contains $P_1$'s, the labels of these edges will be less than the 
labels of the odd edges incident to $v_r$ on the even paths in $R$. 
This could lead to $\varphi_f(v_l)<\varphi_f(v_r)$ if there are too many $P_1$'s in $L$.

Once this happens, our solution is to switch the labeling order of some edges in the even paths in $R$. 
More precisely, we need to change the labeling orders of the edges for  
$\alpha$ even paths, where $\alpha=\max\{0, (b-1)-(c+d)\}$, to construct the desired strongly antimagic labeling.
In addition, we will change the labeling order of the edges on $P_2$'s first. 
Let $b_2$ be the number of $P_2$'s in $R$ and $\beta=\min\{\alpha ,b_2\}$. 
If $\alpha>0$, then $b-1>c+d$.
Since $\deg(v_l)=c+d+t>a+b=\deg(v_r)$ , we have $t>a+1+\alpha> \beta$.

The followings are our labeling rules. Again, some steps can be skipped if no such edges exist.

\noindent{\bf Step 1.} If $\beta>0$, we first label the odd edges of $\beta$ $P_2$'s in $R$ and the edges of $\beta-1$ $P_1$'s in $L$ by
\[f(e^{r,even}_{i,1})=2i-1\hbox{ for } i\in[\beta]\hbox{ and } f(e^{i})=2i \hbox{ for } i\in[\beta -1].\]
Previously, we should label the even edges of $P_2$ in $R$, but now we first label the odd edges of them and leave the pendant edges to be labeled later. 
Observe that for the label of each $e^{r,even}_{i,1}$ is an odd integer, 
and $f(e^i)>f(e^{r,even}_{i,1})$ for $i\in [\beta-1]$. 
We define \[\beta_1=\max\{0,\beta-1\}.\] 

\noindent{\bf Step 2.}
If $\alpha>\beta=b_2$, for $i\in [\beta+1, \alpha]$ and $j\in[4, 2y_i]_{even}$, label the even edges  of the even paths in $R$ by
\begin{center}
$f(e^{r,even}_{i,j})=\beta_1+B^{even}_{i-1}-(i-(\beta+1))+\frac{j-2}{2}$. 
\end{center}
Furthermore, if $b-1>\alpha$, we also label $e^{r,even}_{i,j}$ 
for $i\in [\alpha+1, b-1]$ and $j\in[2y_i]_{even}$ by
\[f(e^{r,even}_{i,j})=\beta_1+B^{even}_{i-1}-(\alpha-\beta)+\frac{j}{2}.\]

\noindent{\bf Step 3.} If $a\ge 1$, for $i\in [a]$ and $j\in[ 2x_i+1]_{odd}$, label the odd edges of odd paths in $R$ by 
\begin{center}
$f(e^{r,odd}_{i,j})=\beta_1+B^{even}_{b-1}-(\alpha-\beta)+A^{odd}_{i-1}+\frac{j+1}{2}$.
\end{center}

\noindent{\bf Step 4.}  If $c\ge 1$, for $i\in [c]$ and $j\in[ 2w_i]_{odd}$, label the odd edges of odd paths in $L$ by 
\begin{center}
$f(e^{l,odd}_{i,j})=\beta_1+B^{even}_{b-1}-(\alpha-\beta)+A^{odd}_{a}+C^{odd}_{i-1}-(i-1)+\frac{j+1}{2}$.
\end{center}
We leave the edges $e^{l,odd}_{i,2w_i+1}$ of odd paths in $L$ to be labeled later to ensure that 
$v_l$ has a large vertex sum.

\noindent{\bf Step 5.} If $s\ge 4$, label the edges of $P^{core}$ by 
\begin{align*}
f(e_j)&=\beta_1+B^{even}_{b-1}-(\alpha-\beta)+A^{odd}_a+C^{odd}_c-c\\
& +\left\{
 \begin{array}{ll}
\frac{s-j}{2}, \hbox{ for }j\in[2, s-2]_{even}, & \hbox{when $s$}\hbox{ is even}. \\
\frac{j-1}{2}, \hbox{ for }j\in[3, s-2]_{odd}, & \hbox{when $s$}\hbox{ is odd}.
                                                        \end{array}
                                                    \right.																										
\end{align*}

Again, we have labeled $s_1=\lfloor\frac{|s-2|}{2}\rfloor$ edges on the core path $P_s$ at this step.
Next, we label the even edges of the $b$-th even path in $R$. 

\noindent{\bf Step 6.} For $j\in[2y_b]_{even}$, label the edges $e^{r,even}_{b,j}$ by
\begin{center}
$f(e^{r,even}_{b,j})=\beta_1+B^{even}_{b-1}-(\alpha-\beta)+A^{odd}_a+C^{odd}_c-c+s_1+\frac{j}{2}.$
\end{center}

\noindent{\bf Step 7.} If $d\ge 1$, for $i\in[d]$ and $j\in[2z_i]_{odd}$, label 
the odd edges of the even paths in $L$ by
\begin{center}
$f(e^{l,even}_{i,j})=\beta_1+B^{even}_{b}-(\alpha-\beta)+A^{odd}_a+C^{odd}_c-c+s_1+D^{odd}_{i-1}+\frac{j+1}{2}.$
\end{center}

\noindent{\bf Step 8.} If $\alpha>\beta=b_2$, for $i\in[\beta+1, \alpha]$, label the edges $e^{r,even}_{i,1}$ in $R$ by
\begin{center}
$f(e^{r,even}_{i,1})=\beta_1+B^{even}_{b}-(\alpha-\beta)+A^{odd}_a+C^{odd}_c-c+s_1+D^{odd}_{d}+(i-\beta)$.
\end{center}
Note that, for $\beta+1\leq i\leq \alpha$, $e^{r,even}_{i,2}$ and $e^{r,even}_{i,3}$ on $P_{2y_i}$ are two incident edges unlabeled yet.

\noindent{\bf Step 9.} If $\beta\ge 1$, then $t>\beta$. 
Recall that we have labeled $\beta-1$ paths of length one in $L$ at Step 1. Now
label the remaining edges $e^i$ in $L$ by
\begin{center}
$f(e^i)=\beta_1+B^{even}_{b}+A^{odd}_a+C^{odd}_c-c+s_1+D^{odd}_{d}+(i-\beta_1)$ for $i\in[\beta, t]$.
\end{center}

\noindent{\bf Step 10.} If $\beta\ge 1$, for $i\in[\beta]$, label the pendant edges of the $P_2$'s in $R$ by
\begin{center}
$f(e^{r,even}_{i,2})=B^{even}_{b}+A^{odd}_a+C^{odd}_c-c+s_1+D^{odd}_{d}+t+(\beta+1-i)$.
\end{center}

\noindent{\bf Step 11.} If $\alpha>\beta$, for $i\in[\beta+1, \alpha]$ and $j\in [3, 2y_i]_{odd}$, label 
the odd edges of the even paths in $R$ by
\begin{center}
$f(e^{r,even}_{i,j})=B^{even}_{b}+A^{odd}_a+C^{odd}_c-c+s_1+D^{odd}_{d}+t+B^{odd}_{i-1}-(i-(\beta+1))+\frac{j-1}{2}$.
\end{center}
Moreover, if $b-1>\alpha$, for $i\in [\alpha+1, b-1]$ and $j\in [2y_i]_{odd}$, label $e^{r,even}_{i,j}$ by
\begin{center}
$f(e^{r,even}_{i,j})=B^{even}_{b}+A^{odd}_a+C^{odd}_c-c+s_1+D^{odd}_{d}+t+B^{odd}_{i-1}-(\alpha-\beta)+\frac{j+1}{2}$.
\end{center}

\noindent{\bf Step 12.} If $a\ge 1$, for $i\in [a]$ and $j\in [2x_i]_{even}$, label the even edges of odd paths in $R$ by
\begin{center}
$f(e^{r,odd}_{i,j})=B^{all}-y_b-(\alpha-\beta)+A^{odd}_a+C^{odd}_c-c+s_1+D^{odd}_{d}+t+A^{even}_{i-1}+\frac{j}{2}$.
\end{center}

\noindent{\bf Step 13.} If $c\ge 1$, for $i\in [c]$ and $j\in [2w_i]_{even}$, label the even edges of odd paths in $L$ by
\begin{center}
$f(e^{l,odd}_{i,j})=B^{all}-y_b-(\alpha-\beta)+A^{all}+C^{odd}_c-c+s_1+D^{odd}_{d}+t+C^{even}_{i-1}+\frac{j}{2}$.
\end{center} 

\noindent{\bf Step 14.} If $s\ge 2$, label the edges in $P^{core}$ by  

\begin{eqnarray*}
           f(e_j) &=& B^{all}-y_b-(\alpha-\beta)+A^{all}+C^{all}-c+s_1+D^{odd}_{d}+t \\
            && +\left\{
\begin{array}{lll}
\frac{s+1-j}{2},& \hbox{ for }  j\in[s]_{odd},& \hbox{when $s$}\hbox{ is even}. \\
\frac{j}{2},& \hbox{ for } j\in[s]_{even},& \hbox{when $s$}\hbox{ is odd}.
\end{array}      
                 \right.
\end{eqnarray*}

\noindent{\bf Step 15.} For $j\in[2y_b]_{odd}$, label the odd edges of the $b$-th even path in $R$ by
\begin{center}
$f(e^{r,even}_{b,j})=B^{all}-y_b-(\alpha-\beta)+A^{all}+C^{all}-c+s-s_2+D^{odd}_{d}+t+\frac{j+1}{2}$.
\end{center}

\noindent{\bf Step 16.} If $d\ge 1$, for $i\in [d]$ and $j\in [2z_i]_{even}$, 
label the even edges of the even paths in $L$ by
\begin{center}
$f(e^{l,even}_{i,j})=B^{all}-(\alpha-\beta)+A^{all}+C^{all}-c+s-s_2+D^{odd}_{d}+t+D^{even}_{i-1}+\frac{j}{2}$.
\end{center}

\noindent{\bf Step 17.} If $\alpha>\beta$, for $i\in[\beta+1, \alpha]$, label the unlabeled edges $e^{r,even}_{i,2}$ by  
\begin{center}
$f(e^{r,even}_{i,2})=B^{all}-(\alpha-\beta)+A^{all}+C^{all}-c+s-s_2+D^{all}+t+(i-\beta)$.
\end{center}

\noindent{\bf Step 18.} If $c\ge 1$, for $i\in [c]$, label $e^{l,odd}_{i,2w_i+1}$ by
\begin{center}
$f(e^{l,odd}_{i,2w_i+1})=B^{all}+A^{all}+C^{all}-c+s-s_2+D^{all}+t+i$.
\end{center}

\noindent{\bf Step 19.}  Label the remaining edges in $P^{core}$ by the following rules: 
If $s=1$ or $s$ is even, then let $f(e_s)=m$; otherwise, let $f(e_1)=m-1$ and $f(e_s)=m$.

Next, we prove that  $f$ is strongly antimagic.

{\bf Claim:} $\varphi_f(u)\ge \varphi_f(v)$ for any $u\in V_2$ and $v\in V_1$.

Observe that all pendent edges have been labeled at Step 9 or Step 10. For the former case, 
$\beta=0$ and there is an unlabeled edge in $E(u)$ for every $u\in V_2$ at the end of Step 8. 
Hence the claim holds.
For the latter case, observe that the edge $e^{r,even}_{1,1}$ has the largest label among all pendent edge. 
Hence the largest vertex sum of all pendent vertices is $f(e^{r,even}_{1,1})$. 
Let us check the vertex sum of a vertex in $V_2$. For $1\le i\le \beta$, the vertex sum at $v^{r,even}_{i,1}=f(e^{r,even}_{i,1})+f(e^{r,even}_{i,2})
=(2i-1)+(B^{even}_{b}+A^{odd}_a+C^{odd}_c-c+s_1+D^{odd}_{d}+t+(\beta+1-i))$ is increasing in $i$.
For any other vertex $u\in V_2$, by our labeling rules, 
we can find one edge $e'\in E(u)$ with $f(e')>f(e^{r,even}_{\beta,1})$ and the other edge $e''\in E(u)$ with $f(e'')>f(e^{r,even}_{\beta,2})$. 
Thus, the smallest vertex sum of a vertex in $V_2$ happens at $v^{r,even}_{1,1}$, and is greater than the vertex sum of any pendent vertex.

{\bf Claim:} $\varphi(u)$ are all distinct for $u\in V_2$.

We have already showed that the vertex sums satisfy
$\varphi_f(v^{r,even}_{1,1})<\varphi_f(v^{r,even}_{2,1})<
\ldots<\varphi_f(v^{r,even}_{\beta,1})< \varphi_f(u)$
for $u\in V_2-\{v^{r,even}_{1,1},v^{r,even}_{2,1},\ldots, v^{r,even}_{\beta,1}\}$.
For other vertices $u'$ and $u''$ in $V_2$, let $E(u')=\{e^1_{u'}, e^2_{u'}\}$ and $E(u'')=\{e^1_{u''}, e^2_{u''}\}$.
Assume $f(e^1_{u'})<  f(e^2_{u'})$ and $f(e^1_{u''})<  f(e^2_{u''})$.
Our labeling rules give that 
if $f(e^1_{u'})\le  f(e^1_{u''})$, then $f(e^2_{u'})\le  f(e^2_{u''})$, and at least 
one of the inequalities is strict. This guarantees that 
$\varphi_f(u)$ are distinct for all $u\in V_2$.

For  $\varphi_f(v_l)>\varphi_f(v_r)>\varphi_f(u)$ for any $u\in V_2$, see Appendix.
\qed

Lemma \ref{deqL equals 3} and Lemma \ref{R has two P1} could be proven 
by the same labeling rules.

\noindent{\bf Proof of Lemma \ref{deqL equals 3} and Lemma \ref{R has two P1}.}
First, We use Lemma \ref{lm1} and Lemma \ref{Right minus 1} to do some reductions.

Given a double spider $DS(L,P^{core}, R)$ described in Lemma \ref{deqL equals 3}, 
let us first consider $h=\min \{j| P_j \in R\cup L\}$. 
Since $\deg(v_l)=\deg(v_r)=3$, without loss of generality, 
we assume the number of $P_h$ in $R$ is greater than or equal to that in $L$.
By Lemma \ref{lm1}, it follows that we only need to show that the double spider is strongly antimagic 
for $h=1$. 

Given a double spider $DS(L,P^{core}, R)$ described in Lemma \ref{R has two P1}, 
we remove all but two $P_1$s in $R$.
Moreover, if there are some $P_1$ in $L$ as well, we remove them as many as possible
unless one of the three situation happens: $L$ contains no $P_1$'s, or 
$L$ consists of exactly two $P_1$'s, 
or $L$ consists of exactly one $P_1$ and one path of length at least two.
By Lemma \ref{Right minus 1}, if the resulting double spider is strongly antimagic,
then $DS(L,P^{core}, R)$ is also strongly antimagic.

Every reduced double spider belongs to at least one of the three types:

\begin{description}
\item{(a)} $\deg(v_l)=\deg(v_r)=3$, $t=2$, $a=2$, and $x_1=x_2=0$; or 
\item{(b)} $\deg(v_l)=\deg(v_r)=3$, $t\le 1$, $a\geq 1$, and $x_1=0$; or  
\item{(c)} $\deg(v_l)\ge  \deg(v_r)=3$, $t=0$, $a=2$,  and $x_1=x_2=0$.
\end{description}

Now we show each type of double spiders above is strongly antimagic.
If a double spider is of type (a), then the total number of edges $m=s+4$.
When $s$ is odd, we give the labeling $f$ as follows:
$f(e_j)=\frac{(s+1)-j}{2}$ for $j\in [s]_{even}$,
$f(e^{r,odd}_{1,1})=\frac{s-1}{2}+1$, $f(e^{r,odd}_{2,1})=\frac{s-1}{2}+2$, 
$f(e^1)=\frac{s-1}{2}+3$, $f(e^2)=\frac{s-1}{2}+4$, and 
$f(e_j)=\frac{s-1}{2}+4+\frac{(s+2)-j}{2}$ for $j\in [s]_{odd}$.
For this labeling, we have vertex sums 
$\varphi_f(v_l)=2s+10$, $\varphi_f(v_r)=\frac{3s+13}{2}$,
$\varphi_f(v_j)=\frac{3s+11-2j}{2}$ for $2\le j\le s$, and 
$\varphi_f(v)\in\{\frac{s+1}{2},\frac{s+3}{2},\frac{s+5}{2},\frac{s+7}{2}\}$ if $\deg(v)=1$.

When $s$ is even, we give the labeling $f$ as follows:
$f(e_j)=\frac{j}{2}$ for $j\in [s]_{even}$,
$f(e^1)=\frac{s}{2}+1$, $f(e^2)=\frac{s}{2}+2$, 
$f(e^{r,odd}_{1,1})=\frac{s}{2}+3$, $f(e^{r,odd}_{2,1})=\frac{s}{2}+4$, 
and $f(e_j)=\frac{s}{2}+4+\frac{j+1}{2}$ for $j\in [s]_{odd}$.
For this labeling, we have vertex sums 
$\varphi_f(v_l)=\frac{3s+20}{2}$, $\varphi_f(v_r)=\frac{3s+10}{2}$,
$\varphi_f(v_j)=\frac{s+8+2j}{2}$ for $j\in[2,s]$, and 
$\varphi_f(v)\in\{\frac{s+2}{2},\frac{s+4}{2},\frac{s+6}{2},\frac{s+8}{2}\}$ if $\deg(v)=1$.

It is easy to see the labelings are strongly antimagic.

For a double spider of type (b) or (c), 
we will give the rules to label the edges by $1,2,\ldots, m$ accordingly. 
Our rules will produce a strongly antimagic labeling 
except for the double spider is isomorphic to the following ones:

\begin{figure}[h]
\begin{center}
\begin{picture}(150,60)
\put(10,10){\circle*{4}}
\put(10,50){\circle*{4}}
\put(-10,50){\circle*{4}}
\put(-30,50){\circle*{4}}
\put(30,30){\circle*{4}}
\put(50,30){\circle*{4}}
\put(70,30){\circle*{4}}
\put(90,10){\circle*{4}}
\put(90,50){\circle*{4}}
\put(30,30){\line(1,0){40}}
\put(10,50){\line(1,-1){20}}
\put(10,50){\line(-1,0){40}}
\put(10,10){\line(1,1){20}}
\put(70,30){\line(1,-1){20}}
\put(70,30){\line(1,1){20}}
\end{picture}
\quad
\begin{picture}(100,60)
\put(10,10){\circle*{4}}
\put(10,50){\circle*{4}}
\put(-10,50){\circle*{4}}
\put(-30,50){\circle*{4}}
\put(30,30){\circle*{4}}
\put(50,30){\circle*{4}}
\put(70,30){\circle*{4}}
\put(90,10){\circle*{4}}
\put(90,50){\circle*{4}}
\put(30,30){\line(1,0){40}}
\put(10,50){\line(1,-1){20}}
\put(10,50){\line(-1,0){40}}
\put(10,10){\line(1,1){20}}
\put(70,30){\line(1,-1){20}}
\put(70,30){\line(1,1){20}}
\put(20,40){7}
\put(19,10){5}
\put(36,32){3}
\put(56,32){8}
\put(74,40){1}
\put(75,10){4}
\put(-4,40){2}
\put(-24,40){6}
\end{picture}
\end{center}
\caption{The double spider with $L=\{P_3,P_1\}$, $R=\{2P_1\}$, $P^{core}=P_2$}\label{fig 2}
\end{figure}
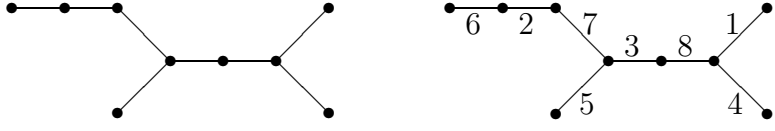

We construct a strongly labeling separately in the right graph of Figure~\ref{fig 2}.

Note that for the two types of double spiders, $R$ contains only two paths and one of them has length one. 
For convenience, we will denote the two paths in $R$ by 
$P_1 =v_rv^{r}_{1,1}=e^{r}_{1,1}$ and $P_k=v_rv^{r}_{2,1}v^{r}_{2,2}\cdots v^{r}_{2,k}$ with   
$e^{r}_{2,j}=v^{r}_{2,j-1}v^{r}_{2,j}$ and $e^{r}_{2,1}=v_rv^{r}_{2,1}$.

The following are our rules to label the double spiders of type (b) and (c):

\noindent{\bf Step 1.} If $k\ge 2$, label all even edges of $P_k$ in $R$ by 
\begin{center}
$f(e^{r}_{2,j})=\lfloor\frac{k+2-j}{2}\rfloor$, for  $j\in [k]_{even}$.
\end{center}

\noindent{\bf Step 2.} If $c\ge 1$,  label all odd edges of $P_{2w_{i}+1}$ in $L$, except for $e^{l,odd}_{1,2w_{1}+1}$, by
\begin{center}
$f(e^{l,odd}_{1,j})=\lfloor\frac{k}{2}\rfloor+\frac{j+1}{2}, \hbox{ for }j\in [2w_{1}-1]_{odd}$,
\end{center} 
and for $i\in[2,c]$ and $j\in[2w_{i}+1]_{odd}$, let
\begin{center}
$f(e^{l,odd}_{i,j})=\lfloor\frac{k}{2}\rfloor+C^{odd}_{i-1}-1+\frac{j+1}{2}$.
\end{center}
Moreover, we define $w'=-1$ when $c\ge 1$, otherwise $w'=0$.
Then we have $\lfloor\frac{k}{2}\rfloor+c+w'+D^{odd}_d+t\ge 1$.

\noindent{\bf Step 3.}
If $s\ge 4$, label the edges of $P^{core}$ by,
\begin{center}
$f(e_j)=\lfloor\frac{k}{2}\rfloor+C^{odd}_c+w'+\left\{
\begin{array}{ll}
\frac{s-j}{2}, \hbox{ for }j\in[2, s-2]_{even}, & \hbox{when $s$}\hbox{ is even}. \\
\frac{j-1}{2}, \hbox{ for }j\in[3, s-2]_{odd}, & \hbox{when $s$}\hbox{ is odd}.
                                                        \end{array}
                                                    \right.$
\end{center}
As before, we labeled $s_1=\lfloor\frac{|s-2|}{2}\rfloor$ edges of the core path $P_s$.

\noindent{\bf Step 4.} If $d\ge 1$, for $i\in[d]$ and $j\in[ 2z_i]_{odd}$,  
label the odd edges of $P_{2z_i}$ by
\begin{center}
$f(e^{l,even}_{i,j})=\lfloor\frac{k}{2}\rfloor+C^{odd}_c+w'+s_1+D_{i-1}^{odd}+\frac{j+1}{2}$.
\end{center}
Next, we label edges of the paths of length one in $L$ and $R$. 
We have to slightly adjust the labeling orders for different cases.  
Let \[t'=\left\{
                          \begin{array}{ll}
                            1, & \hbox{if $t=1,d=1$ or $t=1,s=2,c=1,w_1=1,k\geq 2$;} \\
                            0, & \hbox{otherwise;}
                          \end{array}
                        \right.\]
Observe that if $t'=1$, then $\lfloor\frac{k}{2}\rfloor+D^{odd}_d\ge 1$.

\noindent{\bf Step 5.} We label $e^{1}$ (it does not exsit if $t=0$)
and $e^{r}_{1,1}$ in different order according to the number $t'$.
If $t'=1$, we label $e^{r}_{1,1}$ and $e^{1}$ by
\begin{eqnarray}\label{c11 9}
f(e^{r}_{1,1}) &=& \hbox{$\lfloor\frac{k}{2}\rfloor$}+C^{odd}_c+w'+s_1+D^{odd}_d+1. \\
\notag f(e^{1}) &=& \hbox{$\lfloor\frac{k}{2}\rfloor$}+C^{odd}_c+w'+s_1+D^{odd}_d+2.
\end{eqnarray}
Else, $t'=0$, then we label $e^{1}$ and $e^{r}_{1,1}$ by
\begin{eqnarray}\label{c11 8}
\notag f(e^{1}) &=& \hbox{$\lfloor\frac{k}{2}\rfloor$}+C^{odd}_c+w'+s_1+D^{odd}_d+t. \\
f(e^{r}_{1,1})&=& \hbox{$\lfloor\frac{k}{2}\rfloor$}+C^{odd}_c+w'+s_1+D^{odd}_d+t+1.
\end{eqnarray}
In this step, $f(e^1)$ is undefined when $t=0$.

\noindent{\bf Step 6.} Label all the odd edges of $P_k$, $j\in[k]_{odd}$, in $R$ by
\begin{equation}\label{c11 10}
  f(e^{r}_{2,j})=\hbox{$\lfloor\frac{k}{2}\rfloor+C^{odd}_c+w'+s_1+D^{odd}_d+1+t+\lceil\frac{k+1-j}{2}\rceil$}.
\end{equation}

\noindent{\bf Step 7.} If $c\ge 1$, for $i\in [c]$  and  $j\in [2w_{i}]_{even}$,  label the even edges of $P_{2w_i+1}$ in $L$ by
\begin{center}
$f(e^{l,odd}_{i,j})=k+1+C^{odd}_c+w'+s_1+D^{odd}_d+t+C_{i-1}^{even}+\frac{j}{2}$.
\end{center}

\noindent{\bf Step 8.} If $s\ge 2$, label the edges in $P^{core}$  by 

\begin{equation}\label{c11 12}
  f(e_j)=k+1+C^{all}+w'+D^{odd}_d+s_1+t+\left\{
                                                          \begin{array}{lll}
                                                            \frac{s+1-j}{2},& \hbox{for }j\in[s]_{odd}, & \hbox{when }s\hbox{ is even}; \\
                                                            % &  \\
                                                            \frac{j}{2},& \hbox{for }j\in [s]_{even}, & \hbox{when }s\hbox{ is odd};
                                                          \end{array}\right.
\end{equation}
Let $s_2$ be the number of unlabeled edges on $P^{core}$. So 
$s_2=1$, if $s=1$ or $s$ is even, otherwise $s_2=2$. 

\noindent{\bf Step 9.} If $d\ge 1$, for $j\in [2z_i]_{even}$ and $ i\in[d]$, label the even edges of $P_{2z_i}$ in $L$ by
\begin{center}
$f(e^{l,even}_{i,j})=k+1+C^{all}+w'+s-s_2+D^{odd}_d+t+D_{i-1}^{even}+\frac{j}{2}$.
\end{center}

\noindent{\bf Step 10.} If $c\ge 1$, label the edge $e^{l,odd}_{1,2w_1+1}$ left at Step 2 by 
\begin{center}
$f(e^{l,odd}_{1,2w_1+1})=m-s_2$.
\end{center}

\noindent{\bf Step 11.} Label the remaining edges in $P^{core}$ by the following rules:\\
If $s=1$ or $s$ is even, then let $f(e_s)=m$;
otherwise, let $f(e_1)=m-1$ and $f(e_s)=m$.

We prove $f$ is a strongly antimagic labeling. 

{\bf Claim:} $\varphi_f(v)> \varphi_f(u)$ for any $v\in V_2$ and $u\in V_1$.

Observe that either all pendent edges were labeled 
before Step 6, or there exists exactly one pendent edge labeled at Step 6, 
when $k$ is odd and $k\ge 3$.
In the former case, for every $v \in V_2$, 
there is an edge in $E(v)$ not labeled yet at the 
beginning at Step 6. This promises that $\varphi_f(v)> \varphi_f(u)$ for any 
$u\in V_1$.
In the latter case, we label the pendent edge of $P_k$ by $\lfloor\frac{k}{2}\rfloor+C^{odd}_c+w'+s_1+D^{odd}_d+t+2$
at Step 6, and it is equal to $\varphi_f(v^{r}_{2,k})$.  
Moreover, every vertex $v\in V_2$, except for $v^{r}_{2,k-1}$, is incident 
to an edge of label greater than $\lfloor\frac{k}{2}\rfloor+C^{odd}_c+w'+s_1+D^{odd}_d+t+2$. 
This also leads $\varphi_f(v)> \varphi_f(u)$ for any vertex $v\in V_2$ and $u\in V_1$. 

{\bf Claim:} $\varphi(u)$ are all distinct for $u\in V_2$.

For any two vertices $u'$ and $u''$ in $V_2$, let $E(u')=\{e^1_{u'}, e^2_{u'}\}$ and $E(u'')=\{e^1_{u''}, e^2_{u''}\}$.
Assume $f(e^1_{u'})<  f(e^2_{u'})$ and $f(e^1_{u''})<  f(e^2_{u''})$.
Our labeling rules give that 
if $f(e^1_{u'})\le  f(e^1_{u''})$, then $f(e^2{u'})\le  f(e^2_{u''})$, and at least 
one of the inequalities is strict. This guarantees that 
$\varphi_f(u)$ are distinct for all $u\in V_2$.

For  $\varphi_f(v_l)>\varphi_f(v_r)>\varphi_f(u)$ for any $u\in V_2$, see Appendix.%\hfill$\blacksquare$
\qed

\section{Conclusion and Future Work}

In general, given an antimagic graph $G$, there exist many antimagic labelings on $G$.
Some of the labelings are not strongly antimagic.
Thus, finding a strongly antimagic labeling of a graph could be more difficult
than finding a general antimagic labeling.
In fact, we do not know if there exists a strongly antimagic labeling for every antimagic graph .
However, if a graph is strongly antimagic, then we can use Lemma~\ref{lm1}
to construct a larger graph which is not only antimagic but also strongly antimagic .
It would be helpful to tackle the antimagic labeling problem
if we have more constructive methods like that.
For example, Lemma~\ref{lm1} can be generalized to the following theorem.
\begin{theorem}
Let $G$ be a strongly antimagic graph and $V_k=\{v\in V\mid \deg(v)=k\}$.
If for each vertex in $V_k$, we attach an edge to it,
then the resulting graph is also strongly antimagic.
\end{theorem}
The proof of the above theorem is exactly the same as Lemma~\ref{lm1}. First add $|V_k|$ to
the label of each edge in $E$ when the strongly antimagic labeling is given,
then label the new edges by $1,\ldots, |V_k|$
according to the order of the vertex sums of the vertices in $V_k$.
For antimagic graphs, we ask the following questions.

\begin{question}\label{q1}
Does there exist a strongly antimagic labellings for every antimagic graph?
\end{question}

In 2008 , Wang and Hsiao~\cite{W2008} introduced the {\em $k$-antimagic labeling} on a graph $G$, which
is a bijection $f$ from $E(G)$ to $\{k+1,\ldots,k+|E(G)|\}$ for an integer $k\ge 0$
such that the vertex sums $\varphi_f(v)$ are distinct over all vertices.
We call a graph {\em $k$-antimagic} if it has a $k$-antimagic labeling.
The purpose of studying such kind of labelings is to apply them for
finding the antimagic labelings of the Cartesian product of graphs.
Wang and Hsiao also pointed out that
if the antimagic labeling $f$ of a graph $G$ has the property
that the order of vertex sums is consistent with the order of degrees,
then $G$ is $k$-antimagic for any $k\ge 0$.
This property on the vertex sums is exactly the same definition of the strongly antimagic labeling in our article.
In fact, all the $k$-antimagic labelings studied in~\cite{W2008}
are derived from the strongly antimagic labeling of the graph with a translation on labels.
Hence all those $k$-antimagic labelings have the ``strong property'': $\varphi_f(v)<\varphi_f(u)$ whenever $\deg(u)< \deg(v)$.
\begin{question}\label{q2}
Is there a $k$-antimagic graph but not $(k+1)$-antimagic?
\end{question}
Note that if the answer of Question~\ref{q2} is yes for some graph $G$,
then every $k$-antimagic labeling on $G$ does not have the above strong property on the vertex sums and the degrees.
Moreover, $G$ is a negative answer for Question~\ref{q1} if $k=0$.

{\bf Remark.} There is a different version of $k$-antimagic labeling studied in \cite{Dan2005,W2012}.
They consider injections from $E(G)$ to $\{1,2,...,|E(G)|+k\}$ such that all vertex sums are pairwise distinct.

Recall that the set $V_k$ of a graph consists of vertices of degree $k$.
For any graph, let $V_{\ge 3}$ be the set of vertices of degree at least three.
Kaplan, Lev and Roditty~\cite{KLR2009} proved that for a tree, if the set $|V_2|\le 1$, then it is antimagic.
Our strongly antimagic double spiders together with the known results on spiders and paths
can be rephrased as following: For a tree, if the set $|V_{\ge 3}|\le 2$, then it is antimagic.
If we have both large $V_2$ and $V_{\ge 3}$ in the tree, then the problem turns out to be more difficult. We explain the reasons.
Note that $|V_1|$ must be larger than $|V_{\ge 3}|$ by the simple fact that the average degree
of a tree is less than two.
Hence, large $V_2$ and $V_{\ge 3}$ leads to large $V_2$ and $V_1$.
If we label the edges at random, then the vertex sum of a vertex in $V_2$ has fifty percent likelihood to be smaller than $|E|$,
which is very likely to coincide with the vertex sums of vertices in $V_1$.
A very recently result~\cite{LMS2017} is that for a caterpillar, if $|V_1|\ge \frac{1}{2}(3(|V_2|+|V_{\ge 3}|+1))$, then it is antimagic.
Until the paper is completed, we do not have an affirmative answer of Conjecture~\ref{g2} for all caterpillars yet.

\section*{Acknowledgment}

The first and third authors would like to thank Alfr\'{e}d R\'{e}nyi Institution of Mathematics for host on August, 2017, in Hungary.

\clearpage

\section{Appendix}

\subsection{Rest of the Proof of Lemma \ref{deqL equals 3} and Lemma \ref{R has two P1}. }

{\bf Claim:} $\varphi_f(v_r)>\varphi_f(u)$ for any $u\in V_2$.

Let $u_2$ be the vertex in $V_2$ with the largest vertex sum. 
If $s=1$, we have 
\[f(e^{r}_{1,1})+f(e^{r}_{2,1}) \geq m-1,\]
by Equalities (\ref{c11 9}) or (\ref{c11 8}), and (\ref{c11 10}).
Moreover, $s=1$ implies that we label $m$ to the core edge which incident to $v_r$ and $v_l$.  
So $\varphi_f(u_2)\le (m-1)+(m-2) <2m-1\le f(e_1)+f(e_{1,1}^{r})+f(e_{2,1}^{r})=\varphi_f(v_r)$.

If $s\ge 2$, then $u_2=v_s$ since it is incident to $e_s$, the last labeled edge. 
By Equalities (\ref{c11 10}) and (\ref{c11 12}),
we have 
\[f(e_{s-1})=f(e^{r}_{2,1})+C^{even}_c+
\left\{
                       \begin{array}{ll}
                         1, & \hbox{if $s$ is even;} \\
                         \frac{s-1}{2}, & \hbox{if $s$ is odd;}
                       \end{array}
                     \right.
  \]
	
Recall that $\lfloor\frac{k}{2}\rfloor+c+w'+D^{odd}_d+t>0$.
In Equation (\ref{c11 8}),

we have
\begin{align*}
f(e^{r}_{1,1})&\ge  \lfloor\frac{k}{2}\rfloor+C^{odd}_c+w'+s_1+D^{odd}_d+t+1\\
&=s_1+C^{even}_c+1+(\lfloor\frac{k}{2}\rfloor+c+w'+D^{odd}_d+t)\\
&>s_1+C^{even}_c+1\ge f(e_{s-1})-f(e^r_{2,1}),
\end{align*}
and hence
\begin{center}
$\varphi_f(v_r)=f(e^r_{1,1})+f(e^r_{2,1})+f(e_s)
>f(e_{s-1})+f(e_s)=\varphi_f(u_2).$
\end{center}

\noindent{\bf Claim:} $\varphi_f(v_l)>\varphi_f(v_r)$.

When $t'=1$, we have $f(e^1)=f(e^{r}_{1,1})+1$ by the rules in Step 5. 
Thus, $\varphi_f(v_r)=f(e_s)+f(e^r_{1,1})+f(e^r_{2,1})=(m-1)+f(e^1)+f(e^r_{2,1})$. 
Note that $m-1$ is assigned to an edge $e$ at Step 9, or Step 10, or Step 11.
So,  $e\in E(v_l)\setminus \{e^1\}$.

If $f(e_1)=m-1$ ($s$ is odd and greater than 3), then there exists an edge in $E(v_l)$ 
labeled at Step 9 or 10, 
whose label is $m-2$ and greater than $f(e^r_{2,1})$.
So 
\[\varphi_f(v_l)> (m-1)+f(e^1)+f(e^r_{2,1})=\varphi_f(v_r).\]

If $f(e)=m-1$ for some $e\in E(v_l)-\{e_1,e^1\}$,
then $e_1$ is labeled at Step 11 when $s=1$ or it is labeled at Step 8 when $s$ is even.
In the former case, we have $f(e^r_{2,1})< m-1$ and hence
\[\varphi_f(v_l)\ge f(e^1)+2m-1>f(e^r_{2,1})+f(e^r_{1,1})+m=\varphi_f(v_r).\]
In the latter case, we have $f(e_1)>f(e^r_{2,1})$, and hence 
\[\varphi_f(v_l)\ge f(e^1)+f(e_1)+m-1>f(e^r_{2,1})+f(e^r_{1,1})+m=\varphi_f(v_r).\]

When $t'=0$, recall that $\deg(v_l)=c+d+t+1\ge 3$. 
We classify the possible values of $c$, $d$, and $t$.
\begin{description}
  \item{Case 1.} $c+d\ge 2$. 
	\begin{description}
	\item{Subcase 1.1.} $d\ge 1$.

	Then we can pick two edges $e',e''\in E(v_l)$ labeled at Step 9 and Step 10, whose labels are both greater than 
	$f(e^r_{2,1})$ and $f(e^r_{1,1})$. If $f(e_1)\in\{m,m-1\}$, then
	\[\varphi_f(v_l)\ge f(e')+f(e'')+f(e_1)
	>f(e^r_{2,1})+f(e^r_{1,1})+f(e_s)= \varphi_f(v_r).\] 
	If $f(e_1)\not\in\{m,m-1\}$, then we use $f(e_1)>f(e^r_{1,1})$, and hence
	\[\varphi_f(v_l)\ge (m-1)+(m-2)+f(e_1)>m+f(e^r_{2,1})+f(e^r_{1,1}).\]

	\item{Subcase 1.2.} $d=0$.

If $s=1$, we have $f(e^{l,odd}_{c,2w_c+1})+1=f(e^{r}_{1,1})$. Moreover, 
$f(e^{l,odd}_{1,2w_1+1})- f(e^{r}_{2,1})\geq 2$ since $C^{even}_c\ge 2$.
Therefore, 
\begin{eqnarray*}
\varphi_f(v_l)&\ge& f(e_s)+f(e^{l,odd}_{1,2w_1+1})+f(e^{l,odd}_{c,2w_c+1})\\
              &>&f(e_s)+f(e^{r}_{2,1})+f(e^{r}_{1,1})\\
							&=&\varphi_f(v_r).
\end{eqnarray*}

If $s$ is odd and greater than 1, 
we have $f(e^{l,odd}_{c,2w_c+1})+\frac{s-3}{2}+1=f(e^{r}_{1,1})$ and
$f(e^{l,odd}_{1,2w_1+1}) - f(e^{r,odd}_{2,1})\ge C^{even}_c+\frac{s-1}{2}+1$.
Thus,
\begin{eqnarray*}
\varphi_f(v_l)&\ge& f(e_1)+f(e^{l,odd}_{1,2w_1+1})+f(e^{l,odd}_{c,2w_c+1})\\
               &\ge& (m-1)+[f(e^{r}_{1,1})-(\frac{s-3}{2}+1)]+[f(e^{r,odd}_{2,1})+ C^{even}_c+\frac{s-1}{2}+1]\\
							&>&f(e_s)+f(e^{r}_{2,1})+f(e^{r}_{1,1})\\
							&=&\varphi_f(v_r).
\end{eqnarray*}
If $s$ is even, we have $f(e^{l,odd}_{c,2w_c+1})+\frac{s-2}{2}+1=f(e^{r,odd}_{1,1})$, 
and by Equalities (\ref{c11 10}) and (\ref{c11 12}), we have
$f(e_1)-f(e^{r}_{2,1})=\frac{s}{2}+C^{even}_c\geq \frac{s}{2}+2$.
Then
\begin{eqnarray*}
\varphi_f(v_l)&\ge& f(e_1)+f(e^{l,odd}_{1,2w_1+1})+f(e^{l,odd}_{c,2w_c+1})\\
              &\ge& (f(e^{r}_{2,1})+\frac{s}{2}+2)+[f(e^{r,odd}_{1,1})-(\frac{s-2}{2}+1)]+(m-1)\\
							&>&f(e_s)+f(e^{r}_{2,1})+f(e^{r}_{1,1})\\
							&=&\varphi_f(v_r).
\end{eqnarray*}

\end{description}

  \item{Case 2.} $c+d=1$.

In this case, $t=1$ by the fact $d+c+t=\deg(v_l)\ge 3$ and reduction.
	\begin{description}
	\item{Subcase 2.1.} $c=1$ and $d=0$.
		
Since the special case of $c=1$, $d=0$, $t=1$, $w_1=1$, $s=2$ and $k=1$ has be handled separately
as illustrated in Figure~\ref{fig 2}, 
we may assume at least one of the conditions $w_1\ge 2$, $s\neq 2$, and $k\ge 2$ holds.
Because $t=1$, we have $f(e^1)+1=f(e^{r,odd}_{1,1})$. 

If $s=1$, by Equality (\ref{c11 10}), $f(e^{l,odd}_{1,2w_1+1})- f(e^{r}_{2,1})=w_1+1\geq 2$.
Then 
\[ \varphi_f(v_l)=
f(e^{l,odd}_{1,2w_1+1})+f(e^1)+f(e_1)> f(e^{r}_{2,1})+f(e^{r}_{1,1})+f(e_1)
=\varphi_f(v_r).\]

If $s$ is odd and greater than one, then $f(e_1)=m-1$. By Equalities (\ref{c11 10}) and (\ref{c11 12}), we have
$f(e^{l,odd}_{1,2w_1+1})- f(e^{r}_{2,1}) =w_1+\frac{s+1}{2}\geq 3$. Then 
\[ \varphi_f(v_l)=
f(e^{l,odd}_{1,2w_1+1})+f(e^1)+f(e_1)> f(e^{r}_{2,1})+f(e^{r}_{1,1})+f(e_s)
=\varphi_f(v_r).\]

If $s$ is even, then $f(e^{l,odd}_{1,2w_1+1})=m-1$.
By Equalities (\ref{c11 10}) and (\ref{c11 12}), we have
$f(e_1)- f(e^{r}_{2,1})=w_1+\frac{s}{2}\geq 3.$
Then 
\[ \varphi_f(v_l)=
f(e^{l,odd}_{1,2w_1+1})+f(e^1)+f(e_1)> f(e^{r}_{2,1})+f(e^{r}_{1,1})+f(e_s)
=\varphi_f(v_r).\]

	\item{Subcase 2.2.} $c=0$ and $d=1$.

This cannot happen since $t=1$ and $d=1$ will imply $t'=1$.
\end{description}

\end{description}

\subsection{Rest of the Proof of Lemma~\ref{R has odd path}.}

The conditions $\deg(v_r)\ge 3$ and $b=0$ imply $a\ge 2$.
Without loss of generality, assume the length of the $a$-th odd path in $L$ is at least 3. 
Since $\deg(v_l)\geq 4$, $s\geq 1$, $2x_a+1\geq 3$, and $2x_1+1\geq 1$, 
we have the total number of edges $m\geq D_d^{even}+z_d+7$.

We make some observations.
\begin{itemize}

\item At Step 5, if $t\geq 2$, we have
\begin{eqnarray}
  f(e^t)+f(e^{t-1}) &=& (A^{odd}_a-1+C^{odd}_c-c+s_1+D^{odd}_d+t)\notag \\
	&&+(A^{odd}_a-1+C^{odd}_c-c+s_1+D^{odd}_d+t-1)\notag \\
	&=&(A^{all}+a-2)+(C^{all}-c)+2s_1+D^{all}+(2t-1)\notag \\
  &=& m+a-c+t-(2s_1-s-3)\label{lm6 1}
\end{eqnarray}

\item At Step 9, if $d\geq 2$, we have
\begin{equation}\label{lm6 2}
  f(e^{l,even}_{d,2z_d})\geq m-c-3\hbox{ and }  f(e^{l,even}_{d-1,2z_{d-1}})\geq m-c-z_d-3.
\end{equation}

\item At Step 10, if $s$ is even, then 
          \begin{equation}\label{lm6 3}
            f(e_1)=f(e^{r,odd}_{a,1})-D^{even}_{d}-1,
          \end{equation}
and, by the order we labeled the edges $e^{r,odd}_{a,1}$, $e_{s-1}$, and $e^{l,odd}_{c,2w_c}$, we have
\begin{equation}\label{lm6 4}
  f(e^{r,odd}_{a,1})>f(e_{s-1})>f(e^{l,odd}_{c,2w_c}).
\end{equation}
Moreover, we have 
\begin{equation}\label{lm6 5}
 f(e^{r,odd}_{a,1})=\left\{
                      \begin{array}{ll}
                        m-c-1, & \hbox{if }s=1 \hbox{ or }s\hbox{ is even}. \\
                        m-c-2, & \hbox{otherwise }.
                      \end{array}
                    \right.
\end{equation}

\item At Step 11, if $c\geq 2$, we have
\begin{equation}\label{lm6 6}
  f(e^{l,odd}_{c,2w_c+1})\geq m-2 \hbox{ and } f(e^{l,odd}_{c-1,2w_{c-1}+1})\geq m-3.
\end{equation}

\item At Step 12, if $s$ is odd, $f(e_1)=f(e^{r,odd}_{a,1})+c+1$. With Equality (\ref{lm6 3}), we have
\begin{equation}\label{lm6 7}
  f(e_1)=\left\{
           \begin{array}{ll}
            f(e^{r,odd}_{a,1})+c+1, & \hbox{if $s$ is odd.} \\
            f(e^{r,odd}_{a,1})-D^{even}_{d}-1, & \hbox{if $s$ is even.}
           \end{array}
         \right.
\end{equation}

\end{itemize}

\noindent{\bf Claim:} $\varphi_f(v_r)>\varphi_f(u)$ for any $u\in V_2$.

Let $u_2$ be the vertex of the largest vertex sum in $V_2$.
Then, we have
\[u_2=\left\{
                                             \begin{array}{ll}
                         v_{s}, & \hbox{if $s>1$.} \\
                         v^{l,odd}_{c,2w_c+1}, & \hbox{if $s=1,c>0$.}\\
                         v^{r,odd}_{a,1}, & \hbox{if $s=1, c=0$.}
                                             \end{array}
                                            \right.\]
By Inequality (\ref{lm6 4}) and $f(e_s)=m>f(e^{l,odd}_{c,2w_c+1})>f(e^{r,odd}_{a,2})$,
the vertex sum at $v_r$ is 
\begin{eqnarray*}
  \varphi_f(v_r) &= &\sum_{i=1}^{a-1}f(e^{r,odd}_{i,1})+f(e^{r,odd}_{a,1})+f(e_s)>f(e^{r,odd}_{a,1})+f(e_s)\\
                 &> &\left\{
         \begin{array}{ll}
         f(e_{s-1})+f(e_s)=\varphi_f(v_{s}), & \hbox{if $s>1$;} \\
   f(e^{l,odd}_{c,2w_c})+f(e^{l,odd}_{c,2w_c+1})=\varphi_f(v^{l,odd}_{c,2w_c+1}), & \hbox{if $s=1,c>0$;}\\
   f(e^{r,odd}_{a,1})+f(e^{r,odd}_{a,2})=\varphi_f(v^{r,odd}_{a,1}), & \hbox{if $s=1,c=0$;}\end{array}\right.  \\
                 &= &\varphi_f(u_2). 
\end{eqnarray*}

\noindent{\bf Claim:} $\varphi_f(v_l)>\varphi_f(v_r)$.

Recall that $\deg(v_l)>\deg (v_r)\ge 3$, 
and for any $e\in E(v_l)$, we have $f(e)>f(e^{r,odd}_{i,1})$ for $1\leq i\leq a-1$. 
Thus, if we can find three edges in $E(v_l)$ such that the sum of the labels is 
not less than the sum of of the maximal two labels of the edges in $E(v_r)$, namely $f(e_s)+f(e^{r,odd}_{a,1})$, 
then we are done.
Recall that $\deg(v_l)=c+d+t+1$. The choice of the three edges in $E(v_l)$ 
depends on the values of $c$, $d$, and $t$:
\begin{description}
  \item{Case 1.} $c\geq 2$

	By Inequality (\ref{lm6 6}) and Equality (\ref{lm6 7}), and $m\geq D^{even}_d+z_d+7$, 
\begin{eqnarray*}
f(e^{l,odd}_{c,2w_c+1})+f(e^{l,odd}_{c-1,2w_{c-1}+1})+f(e_1)
   & \geq & (m-2)+(m-3)\\
	&&+(f(e^{r,odd}_{a,1})-D^{even}_d-1)\\
   &=& (m+f(e^{r,odd}_{a,1}))+(m-D^{even}_d-6) \\
   &>& f(e_s)+f(e^{r,odd}_{a,1}).
\end{eqnarray*}
 
  \item{Case. 2} $c=1$
\begin{description}
  \item{Subcase 2.1.} $d\geq 1$. 
	
	By Inequalities (\ref{lm6 2}) and (\ref{lm6 6}), and Equality (\ref{lm6 7}),
\begin{eqnarray*}
    f(e^{l,odd}_{c,2w_c+1})+f(e^{l,even}_{d,2z_{d}})+f(e_1)
   &\geq& (m-2)+(m-c-3)\\
	&&+(f(e^{r,odd}_{a,1})-D^{even}_d-1) \\
   &=& (m+f(e^{r,odd}_{a,1}))+(m-D^{even}_d-7) \\
  &>& f(e_s)+f(e^{r,odd}_{a,1}).
\end{eqnarray*}
\item{Subcase 2.2.} $d=0$. 

By Inequality (\ref{lm6 6}) and Equality (\ref{lm6 7}),
\begin{eqnarray*}
  f(e^{l,odd}_{c,2w_c+1})+f(e^t)+f(e_1) 
   &\geq& (m-2)+(A^{odd}_a-1+C^{odd}_c-c+s_1+D^{odd}_d+t) \\
    &&   +(f(e^{r,odd}_{a,1})-D^{even}_d-1)\\ 
	&=& (m+f(e^{r,odd}_{a,1}))+(A^{odd}_a+C^{odd}_c+s_1+t-5) \\
   &>& f(e_s)+f(e^{r,odd}_{a,1}).
\end{eqnarray*}

\end{description}

\item{Case 3.} $c=0$.
\begin{description}

 \item{Subcase 2.1.} $d\geq 2$.

 By Inequality (\ref{lm6 2}) and Equality (\ref{lm6 7}),
\begin{eqnarray*}
    f(e^{l,even}_{d,2z_d})+f(e^{l,even}_{d-1,2z_{d-1}})+f(e_1) 
   &\geq& (m-3)+(m-z_d-3)\\
	&&+(f(e^{r,odd}_{a,1})-D^{even}_d-1) \\
  &=& (m+f(e^{r,odd}_{a,1}))+(m-D^{even}_d-z_d-7)\\
   &\ge& f(e_s)+f(e^{r,odd}_{a,1}).
\end{eqnarray*}
  \item{Subcase 2.2.} $d=1$.
	
	Then we have  $t\geq 2$. By Inequality (\ref{lm6 2}), Equality (\ref{lm6 7}), and $A^{odd}_a\geq 3$, 
\begin{eqnarray*}
    f(e^{l,even}_{d,2z_d})+f(e^t)+f(e_1) 
   &\geq& (m-3)+(A^{odd}_a-1+C^{odd}_c-c+s_1+D^{odd}_d+t)\\
	 & &+(f(e^{r,odd}_{a,1})-D^{even}_d-1) \\
   &\geq& (m+f(e^{r,odd}_{a,1}))+(A^{odd}_a+s_1+t-5) \\
   &\geq& f(e_s)+f(e^{r,odd}_{a,1}).
\end{eqnarray*}
  \item{Subcase 2.3.} $d=0$
	
	Then $t\geq 3$. If $s$ is odd, by Equalities (\ref{lm6 1}) and (\ref{lm6 7}), and $a\geq 2$,
\begin{eqnarray*}
  f(e^t)+f(e^{t-1})+f(e_1)  
   &\geq& m+t+a-c-6+(f(e^{r,odd}_{a,1})+c+1) \\
   &=& (m+f(e^{r,odd}_{a,1}))+(t+a-5) \\
   &\geq& f(e_s)+f(e^{r,odd}_{a,1}).
\end{eqnarray*}
If $s$ is even, by Equality (\ref{lm6 1}), 
\begin{eqnarray*}
   f(e^t)+f(e^{t-1})+f(e_1) 
   &\geq& m+t+a-c-5+(f(e^{r,odd}_{a,1})-D^{even}_d-1) \\
  &=& (m+f(e^{r,odd}_{a,1}))+(t+a-6)\\
	&\ge &f(e_s)+f(e^{r,odd}_{a,1})+(t+a-6).
\end{eqnarray*}
The quantity $t+a-6$ in the above inequality is negative only if $t=3$ and $a=2$. However,  
 we have $f(e^1)\ge (A^{odd}_a-1)+1 \geq 3$  and $f(e^{r,odd}_{1,1})=1$. So
\begin{eqnarray*}
  \varphi_f(v_l) &=& f(e^1)+f(e^{t})+f(e^{t-1})+f(e_1) \\
   &\geq& 3+(m+f(e^{r,odd}_{a,1}))+(t+a-6) \\
   &>&1+f(e_s)+f(e^{r,odd}_{a,1})\\
   &=&f(e^{r,odd}_{1,1})+f(e_s)+f(e^{r,odd}_{a,1})\\
   &=&\varphi_f(v_r).
\end{eqnarray*}

\end{description}

\end{description}

\subsection{Rest of the Proof of of Lemma \ref{R has even path}. }

We make some observations.
\begin{itemize}

\item From Step 1, Step 8, and Step 9, we have  
 \begin{equation}\label{lm9b18}
 f(e^i)>f(e^{r,even}_{i,1})\hbox{ for }i\in[\alpha], \hbox{ and } f(e^i)>f(e^{r,odd}_{i',1})\hbox{ for } i\in [\beta, t]\hbox{ and } i'\in[a].
    \end{equation}

\item From Step 16 and Step 18, we have
\begin{equation}\label{lm9b17}
          f(e)>f(e').
        \end{equation}
for $e\in \{e^{l,odd}_{i,2w_i+1},i\in [c]\}\cup \{e^{l,even}_{i,2z_i}, i\in [d]\}$ and $e'\in E(v_r)\setminus \{e_s\}$.

\item At Step 9, if $t\geq 2$,  we have
\begin{equation}\label{lm9b11}
  f(e^t)=B^{even}_{b}+A^{odd}_a+C^{odd}_c-c+s_1+D^{odd}_{d}+t,\hbox{ if }t\geq 1,
\end{equation}
 and
\begin{equation}\label{lm9b12}
 f(e^{t-1})=B^{even}_{b}+A^{odd}_a+C^{odd}_c-c+s_1+D^{odd}_{d}+t-1,\hbox{ if }t\geq 2.
\end{equation}

\item At Step 14, if $s$ is even, then  we have
\begin{equation}\label{lm9b13}
  f(e_1)=m-y_b-(\alpha-\beta)-c-D^{even}_d-1.
\end{equation}

\item At Step 15, after labeling $e^{r,even}_{b,1}$, we have
\begin{equation}\label{lm9b14}
  f(e^{r,even}_{b,1})=m-D^{even}_d-(y_b-1)-(\alpha-\beta)-c-\left\{
                                                   \begin{array}{ll}
                                                     1, & \hbox{if }s=1 \hbox{ or }s\hbox{ is even}, \\
                                                     2, & \hbox{if }s\ge 3\hbox{ and is odd }.
                                                   \end{array}
                                                 \right.
\end{equation} 
Moreover, when $s\ge 2$
\begin{equation}\label{lm9b15}
  f(e^{r,even}_{b,1})>f(e_{s-1}).
\end{equation}

\item By the order we labeled edges on $E$, we have
\begin{equation}\label{lm9b16}
  f(e^{r,even}_{b,1})>f(e^{l,odd}_{c,2w_c})>f(e^{r,even}_{\alpha,3})>f(e^{l,even}_{d,2z_d-1})>f(e^{r,even}_{b,2y_b}).
\end{equation}

\end{itemize}

\noindent{\bf Claim:} $\varphi_f(v_r)>\varphi_f(u)$ for any $u\in V_2$.

Let $u_2$ be the vertex in $V_2$ with the largest vertex sum. If $s=1$, then 
\[\varphi_f(u_2)=\left\{
   \begin{array}{ll}
  \varphi_f(v^{l,odd}_{c,2w_c+1})=f(e^{l,odd}_{c,2w_c})+f(e^{l,odd}_{c,2w_c+1}), & \hbox{if $c>0$;} \\
  \varphi_f(v^{r,even}_{\alpha,2})=f(e^{r,even}_{\alpha,3})+f(e^{r,even}_{\alpha,2}), & \hbox{if $c=0,(\alpha-\beta)>0$;} \\
  \varphi_f(v^{l,even}_{d,2z_d})=f(e^{l,even}_{d,2z_d-1})+f(e^{l,even}_{d,2z_d}), & \hbox{if $c=0,(\alpha-\beta)=0, d>0$;} \\
  \varphi_f(v^{r,even}_{b,2y_b-1})=f(e^{r,even}_{b,2y_b})+f(e^{r,even}_{b,2y_b-1}), & \hbox{otherwise.}
                                                          \end{array}
                                                        \right.\]
By Inequality (\ref{lm9b16}) and $f(e_s)=m$, we have \[\varphi_f(v_r)>f(e^{r,even}_{b,1})+f(e_s)>\varphi_f(u_2).\]
If $s\ge 2$, then $u_2=v_s$. By Inequality (\ref{lm9b15}), we have $\varphi_f(v_r)>\varphi_f(u_2)$.

\noindent{\bf Claim:} $\varphi_f(v_l)>\varphi_f(v_r)$.

The idea is similar to that in the proof of Lemma~\ref{R has odd path}. 
We will choose $k+1$ edges in $E(v_l)$ and $k$ edges in $E(v_r)$ 
such that the sum of the labels of the $k+1$ edges in $E(v_l)$ is not less than
the sum of the labels of the $k$ edges in $E(v_r)$.
Moreover, for other edges $e'\in E(v_l)$ and $e''\in E(v_l)$ which are not chosen,  $f(e')>f(e'')$ holds.  
\begin{description}
  \item{Case 1.} $s=1.$ 

If $t\le 1$, by Inequalities (\ref{lm9b18}) and (\ref{lm9b17}), we have $\varphi_f(v_l)>\varphi_f(v_r)$. 

If $t\geq 2$,
by Equalities (\ref{lm9b11}), (\ref{lm9b12}), and (\ref{lm9b14}), we have \[f(e^{t})+f(e^{t-1})=m+a-c+t-2> f(e^{r,even}_{b,1}).\]
With Inequalities (\ref{lm9b18}) and (\ref{lm9b17}), $\varphi_f(v_l)>\varphi_f(v_r)$.

 \item{Case 2.} $s\ge 2.$

If $t=0$, we have either $f(e^{l,odd}_{c,2w_c+1})+f(e_1)>f(e_s)$ or $f(e^{l,even}_{d,2z_d})+f(e_1)>f(e_s)$; and
if $t=1$, by Equalities (\ref{lm9b11}) and (\ref{lm9b13}), $f(e^t)+f(e_1)>f(e_s)$ holds.
With Inequalities (\ref{lm9b18}) and (\ref{lm9b17}), we have $\varphi_f(v_l)>\varphi_f(v_r)$.

For $t\geq 2,$ note $f(e^t)+f(e^{t-1})>f(e^{r,even}_{b,1})$ and if $s$ is odd, then $f(e_1)=m-1$. 
So $f(e^t)+f(e^{t-1})+f(e_1)\geq f(e_s)+f(e^{r,even}_{b,1})$. 
With Inequalities (\ref{lm9b18}) and (\ref{lm9b17}), we have $\varphi_f(v_l)>\varphi_f(v_r)$.
If $s$ is even, we need compare more edges.
First we have $f(e_1)=f(e^{r,even}_{b,1})+1$ and,
by Equalities (\ref{lm9b11}) and  (\ref{lm9b12}), $f(e^t)+f(e^{t-1})=m+a-c+t-3$.

\begin{description}
  \item{Subcase 2.1.} $c\ge 1$. 
	
We compare the sum of the labels of the edges 
$e^{l,odd}_{c,2w_c+1}$, $e^t$, $e^{t-1}$, an $e_1$ in $E(v_l)$
 and the sum of maximal three labels of edges in $E(v_r)$. 
Let 
\[r=\max\{f(e)\mid e\in E(v_r)\setminus \{e_s,e^{r,even}_{b,1}\}\}.\] 
Then $f(e^{l,odd}_{c,2w_c+1})-r>c+3$, and hence 
\[
f(e^{l,odd}_{c,2w_c+1})+f(e^t)+f(e^{t-1})+f(e_1)>m+f(e^{r,even}_{b,1})+r.
\]
With Inequalities (\ref{lm9b18}) and (\ref{lm9b17}), $\varphi_f(v_l)>\varphi_f(v_r)$ holds.

\item{Subcase 2.2.} $c=0$.
	
If $t>3$ or $a\geq 1$, then $m+a-c+t-3\geq m+1$ and $f(e^t)+f(e^{t-1})+f(e_1)\geq f(e_s)+f(e^{r,even}_{b,1})$. 
The remaining cases are $t=2$ and $a=0$, or $t=3$ and $a=0$. 
Note that $a=0$ implies $b\ge 2$, because $\deg(v_r)\ge 3$.
If $d>0$, no matter $t=2$ or $t=3$, 
$f(e^{l,even}_{d,2z_d})-f(e^{r,even}_{b-1,1})>2$. 
So \[f(e^{l,even}_{d,2z_d})+f(e^t)+f(e^{t-1})+f(e_1)>m+f(e^{r,even}_{b,1})+f(e^{r,even}_{b-1,1}).\]
With Inequalities (\ref{lm9b18}) and (\ref{lm9b17}), $\varphi_f(v_l)>\varphi_f(v_r)$.

If $d=0$, by $\deg(v_l)>\deg(v_r)$, we have $t=3$ and $b=2$.  Hence $\beta=1$. Then 
$f(e^{t-2})=f(e^1)=3$ and $f(e^{r,even}_{1,1})=1$. 
Therefore, 
\begin{eqnarray*}
\varphi_f(v_l)&=&f(e^t)+f(e^{t-1})+f(e^{t-2})+f(e_1)\\
              &>&m+f(e^{r,even}_{b,1})+f(e^{r,even}_{b-1,1})\\
							&=&\varphi_f(v_r).
\end{eqnarray*}

\end{description}

\end{description}


\begin{thebibliography}{99}

\bibitem{AKLRY2004} N. Alon, G. Kaplan, A. Lev, Y. Roditty, and R.
Yuster, {\em Dense graphs are antimagic}, J. Graph Theory,
47 (2004), 297-309.
\bibitem{BBV2015} K. B\'{e}rczi, A. Bern\'{a}th, and M. Vizer,{\em Regular graphs are antimagic}, The Electronic Journal of Combinatorics 22 (2015), paper P3.34
\bibitem{CLPZ2016} F. Chang, Y.-C. Liang, Z. Pan, X. Zhu, {\em Antimagic labeling of regular graphs}, J. Graph Theory 82 (2016), 339-349.

\bibitem{Cra2009} D. W. Cranston, {\em Regular bipartite graphs are antimagic},
J. Graph Theory 60 (2009), 173-182.





\bibitem{HR1990} N. Hartsfield and G. Ringel, {\em Pearls in Graph Theory}, Academic
Press, INC., Boston, 1990, pp. 108-109, Revised version 1994.

\bibitem{Dan2005} D. Hefetz, {\em Anti-magic graphs via the Combinatorial Nullstellensatz}, J Graph Theory 50 (2005), 263-272



\bibitem{KLR2009} G. Kaplan, A. Lev and Y. Roditty, {\em On zero-sum
partitions and antimagic trees}, Discrete Math. 309 (2009), 2010-2014.



\bibitem{LMS2017} A. Lozano, M. Mora and C. Seara, {\em Antimagic Labeling of Caterpillas}, ArXiv1708.00624


\bibitem{LWZ2012} Y. Liang, T. Wong and X. Zhu, {\em Anti-magic labeling of trees}, Discrete Math. 331 (2014), 9-14.




\bibitem{H2015} T.-Y. Huang, {\em Antimagic Labeling on Spiders}, Master Thesis, Department of Mathematics, National Taiwan University.(2015)

\bibitem{S2015} J.-L. Shang, {\em Spiders are antimagic}, Ars Combinatoria, 118 (2015), 367-372.

\bibitem{W2012} T. Wong and X. Zhu. {\em Antimagic labelling of vertex weighted graphs}, Journal of Graph Theory. 70(3), (2012), 348�V350.


\bibitem{W2008} T.-M. Wang and C. C. Hsiao, {\em On anti-magic labeling for graph products},  Discrete Math. 308(16), (2008), 3624�V3633.

\end{thebibliography}
\end{document}